\documentclass[a4paper,reqno]{amsart}

\usepackage{amssymb,amsxtra}
\usepackage{graphicx} \usepackage{amsmath} \usepackage{amssymb}
\usepackage[svgnames]{xcolor}
\usepackage[color,matrix,arrow]{xy}
\usepackage[colorlinks,
                linkcolor = blue,
                anchorcolor =red,
                citecolor = green]{hyperref}

\input xy \xyoption{all}

\newtheorem{theorem}{Theorem}[section]
\newtheorem{lemma}[theorem]{Lemma}
\newtheorem{corollary}[theorem]{Corollary}
\newtheorem{proposition}[theorem]{Proposition}

\theoremstyle{definition}
\newtheorem{definition}[theorem]{Definition}

\newtheorem{remark*}[]{Remark}
\newtheorem{example}[theorem]{Example}
\newtheorem{example*}[]{Example}

\newcommand{\Hom}{\mathrm{Hom}}
\newcommand{\End}{\mathrm{End}}

\newcommand{\rad}{\mathrm{rad}\,}
\newcommand{\uhom}[1]{\underline{\mathrm{Hom}}}
\newcommand{\Ext}{\mathrm{Ext}}
\newcommand{\Ker}{\mathrm{Ker}\,}
\newcommand{\Img}{\mathrm{Im}\,}
\newcommand{\add}{\mathrm{add}\,}
\newcommand{\Tr}{\mathrm{Tr}\,}
\newcommand{\pd}{\mathrm{pd}\,}
\newcommand{\id}{\mathrm{id}\,}

\newcommand{\C}{\mathcal{C}}

\begin{document}
\title{Tensor products of higher APR Tilting Modules}
\author[Xiaojian Lu]{Xiaojian Lu\\ LCSM(Ministry of Education), School of Mathematics and Statistics\\ Hunan Normal University\\ Changsha, Hunan 410081, P. R. China}

\thanks{E-mail: xiaojianlu10@126.com
}
\date{}

\begin{abstract}
 The higher APR tilting modules and higher BB tilting modules were introduced and studied in higher Auslander-Reiten theory. Our objective is to consider these tilting modules by the corresponding simple modules, and show that the tensor product of higher APR (BB) tilting modules is a higher APR (BB) tilting module.
\end{abstract}

\maketitle

\section{Introduction}\label{Intr}

Throughout this paper, assume that $n,m$ are two positive integers and $K$ is a field.

The higher-dimensional Auslander-Reiten theory as a generalization of classical Auslander-Reiten theory \cite{ARS,ASS} was introduced by Iyama and his coauthors \cite{I07-1,I07-2,I08-1,I11} and developed by many authors \cite{DI,IO12,HIO14,AIR,G16,GLL}.
In this setting, the classical tilting theory \cite{ASS,H88} were generalized to higher-dimensional analogs,
the BB tilting modules \cite{BB80} were generalized to the $n$-BB tilting modules by Hu and Xi in \cite{HX} which can be used to construct $n$-almost split sequences,
the APR tilting modules \cite{APR79} were generalized to the $n$-APR tilting modules by Iyama and Oppermann in \cite{IO12} which is the special $n$-BB tilting modules.
The more general tilting modules have been presented in \cite{BO4,M86}.

%$n$-representation-finite algebras and $n$-representation-infinite algebras as a generalization of hereditary algebras has been %studied in higher Auslander-Reiten theory.
Many scholars have studied the $n$-APR tilting modules which plays an important role in higher Auslander-Reiten theory.
Let $\Lambda$ be an $n$-representation-finite algebra or $n$-representation-infinite algebra, in \cite{IO12,HIO14,MY16}, they pointed that any simple projective and non-injective $\Lambda$-modules $P$ admits the $n$-APR tilting $\Lambda$-module associated with $P$,
moreover $n$-APR tilting modules preserve $n$-representation finiteness and $n$-representation infiniteness.
 Mizuno in \cite{M14} provided the description of quivers with relations of $n$-APR tilts.
In 2021, under certain condition, Guo and Xiao showed in \cite{GC21} that the $n$-APR tilts of the quadratic dual of truncations of $n$-translation algebras are realized as $\tau$-mutations.

The tensor product is a very effective research tool in representation theory of finite dimension algebras \cite{Aus55,CC17,L76,Le94,P-1,P-2}.
For $n$-, $m$-representation-finite algebras $\Lambda$, respectively $\Gamma$ over perfect field $K$, under condition of $l$-homogeneity, Herschend and Iyama in \cite{HI11} showed that tensor product $\Lambda\otimes_K \Gamma$ is an $(n+m)$-representation finite algebra which admits the $(n+m)$-APR tilting $(\Lambda\otimes_K \Gamma)$-module associated with simple projective module.
 In this case, it is a natural question to discuss the relationship between the $n$-, $m$-APR tilting modules over $\Lambda$, respectively $\Gamma$ and the $(n+m)$-APR tilting modules over algebra $\Lambda\otimes_K \Gamma$.
 There is a similar question for higher representation-infinite algebras by \cite{HIO14,MY16}.

In this paper, we study this question for general algebras.
The main tool is tensor products over field.
Given two finite dimension algebras $\Lambda$ and $\Gamma$ over field $K$ admitting $n$-respectively, $m$-APR tilting modules or $n$-respectively, $m$-BB tilting modules, we show that how to construct the $(n +m)$-APR tilting modules or $(n +m)$-BB tilting modules over the tensor product algebra $\Lambda\otimes_K \Gamma$.
Precisely,
under certain conditions, we prove that tensor product of $n$-BB tilting module with $m$-BB tilting module is an $(n+m)$-BB tilting module (see Theorem \ref{prop-n-BB-and-finite-global-dimensions}), tensor product of $n$-APR tilting module with $m$-APR tilting module is an $(n+m)$-APR tilting module (see Theorem \ref{the-tensor-products-of-n-APR-tilting-modules}).
As an application, we give a description of the higher APR tilting modules over the tensor products of higher hereditary algebras (see Corollary \ref{coro-APR-over-tensor-products-of-higher-hereditary}).
Furthermore, we give a characterisation of $\tau_{n}$-finite algebras by tensor products (see Theorem \ref{them-Tensor-products-and-tau{n}-finite}).

The article is organized as follows.
In the Section \ref{Prel}, we recall the definition of $\tau_{n}$-finite algebras and higher tilting modules.
In the Section \ref{Prep}, we study the modules and complexes over tensor product algebras and give some preparation results.
In the Section \ref{Tensor-products-of-higher-tilting-modules}, we investigate the tensor products of higher tilting modules and $\tau_{n}$-finite algebras, give the proof of main results. Moreover, let $\Lambda,\Gamma$ be basic ring-indecomposable $n$-respectively $m$-hereditary algebras, we give a description of the relationship between the $n$-, $m$-APR tilting modules over $\Lambda$, respectively $\Gamma$ and the $(n+m)$-APR tilting modules over algebra $\Lambda\otimes_K \Gamma$.

\section{Preliminaries}\label{Prel}

Throughout this paper,  $K$ is a field, all algebras are associative, unital, and finite dimensional over field $K$.
Let $\Lambda$ be a finite dimensional algebra over $K$ and $n$ a positive integer, we denote by $\rad \Lambda$ the Jacobson radical of $\Lambda$ and by $\mod\Lambda$ the category of the finitely generated left $\Lambda$-modules.
   We denote by $\Lambda^{op}$ the opposite algebra of $\Lambda$ and by $D = \Hom_{K}(-, K):\mod \Lambda \longrightarrow  \mod \Lambda^{op}$ the standard $K$-duality.
For $p \geq 0$, $\Ext^{p}_{\Lambda}(-,-)$ is the p{\rm th} extension bifunctor.

All tensor products $\otimes$ are over ${K}$.
Let $\Lambda,\Gamma$ be two finite dimensional algebras over $K$, then the ${K}$-module $\Lambda \otimes_{K} \Gamma$  becomes a finite dimensional algebra over $K$ with multiplication $(a_{1} \otimes b_{1})(a_{2} \otimes b_{2}) =a_{1}a_{2} \otimes b_{1}b_{2}$ for $a_{1},a_{2} \in \Lambda$, $b_{1},b_{2} \in \Gamma$, moreover $(\Lambda \otimes \Gamma)^{op} \cong \Lambda^{op}\otimes \Gamma^{op}$.
Let $M$, $N$ be left $\Lambda$-respectively $\Gamma$-modules, $M \otimes_{K} N$ was converted into a left $(\Lambda \otimes_{K} \Gamma)$-module in such a way that
$(a \otimes b)(m \otimes n) =am \otimes bn$ for $a \in \Lambda$, $b \in \Gamma$, $m \in M$,  $n \in N$ (see \cite[\uppercase\expandafter{\romannumeral9}] {CE}).

 The $n$-Auslander-Reiten translations of $\Lambda$-modules are introduced by Iyama (see \cite{I07-1,I08-1,I11}),
\begin{align*}
 \tau_{n} =D\Tr\Omega^{n-1}: \mod \Lambda \longrightarrow  \mod \Lambda,
  ~\tau_{n}^{-}=\Tr\Omega^{n-1}D: \mod \Lambda \longrightarrow  \mod \Lambda.
  \end{align*}
When the global dimension ${\rm gl.dim} \Lambda \leq n$, $\tau_{n}$ and $\tau_{n}^{-}$ are induced by
\begin{align*}
 \tau_{n}=D\Ext_{\Lambda}^{n}(-, \Lambda):\mod \Lambda \longrightarrow  \mod \Lambda,
  \tau_{n}^{-}=\Ext_{\Lambda^{op}}^{n}(D-, \Lambda):\mod \Lambda \longrightarrow \mod \Lambda.
 \end{align*}

\subsection{tau{n}-finite algebras}
The $\tau_{n}$-finite algebras was studied in \cite{I11}.
\begin{definition}\label{def-tau{n}-finite}
Assume that $\Lambda$ is a finite dimensional algebra over field $K$ and $n\geq 1$. We say that $\Lambda$ is $\tau_{n}$-finite if global dimensions ${\rm gl.dim} \Lambda \leq n$ and $\tau_{n}^{l}(D\Lambda)=0$ holds for some positive integer $l$, $\Lambda$ is $\tau_{n}$-infinite if global dimensions ${\rm gl.dim} \Lambda \leq n$ and $\tau_{n}^{l}(D\Lambda) \neq 0$ for any positive integer $l$.
\end{definition}

\begin{lemma}\label{lem-tau{n}-finite}\cite{I11}
A finite-dimensional algebra $\Lambda$ is $\tau_{n}$-finite if and only if $\Lambda^{op}$ is $\tau_{n}$-finite.
\end{lemma}

As the algebras of the global dimensions at most $n$, $n$-complete algebras and $n$-representation infinite algebra was studied in higher representation theory.
\begin{example}\label{exam-tau{n}-finite}
\begin{enumerate}
\item Any $n$-complete algebra is $\tau_{n}$-finite (see \cite[Proposition 1.12] {I11}).

\item Any $n$-representation infinite algebra is $\tau_{n}$-infinite.  Since for $n$-representation infinite algebra $\Lambda$, the functor $\nu_{n}$ is an auto-equivalence of the bounded derived category of $\mod \Lambda$ (see \cite{HIO14}), so $\nu_{n}^{i}(D\Lambda) \neq 0$ for $i \geq 0$. By \cite[Proposition 4.21] {HIO14}, we have $\tau_{n}^{i}(D\Lambda) =\nu_{n}^{i}(D\Lambda) \neq 0$ for $i \geq 0$.
\end{enumerate}
\end{example}

\subsection{Higher tilting modules}

From the viewpoint of higher representation theory, we consider higher tilting modules.
As a generalization of APR tilting modules \cite{APR79}, the $n$-APR tilting modules was introduced by Iyama and Oppermann (see \cite{IO12}).
As a generalization of BB tilting modules \cite{BB80}, the $n$-BB tilting modules was studied by Hu and Xi (see \cite{HX}).

\begin{definition}\label{def-n-APR-tilting-module}
Suppose that  $\Lambda$ is a basic finite dimensional algebra and $n \geq 1$. Suppose $P$ is a simple projective $\Lambda$-module. We decompose $\Lambda = P \oplus Q$ as a $\Lambda$-module. If $P$ satisfies $\Ext^{i}_{\Lambda}(D \Lambda, P) = 0$ for any $0 \leq i < n$, then we call
$$T = (\tau_{n}^{-}P) \oplus Q$$
the $weak$ $n$-$APR$ $tilting$ $module$ associated with $P$.
If moreover injective dimension $\id_{\Lambda} P = n$, then we call $T$ an $n$-$APR$ $tilting$ $module$ and we call $\End_{\Lambda}(T)^{op}$ an $n$-$APR$ $tilt$ of $\Lambda$.
Dually we define $(weak)$ $n$-$APR$ $cotilting$ $modules$ and $n$-$APR$ $cotilt$ of $\Lambda$.
\end{definition}

By above Definition, an $n$-APR tilting module $T = (\tau_{n}^{-}P) \oplus Q$ associated with $P$ implies that $P$ is an simple projective and non-injective $\Lambda$-module.

\begin{definition}\label{def-n-BB-tilting-module}
Suppose that  $\Lambda$ is a basic finite dimensional algebra and $n \geq 1$. Suppose $S$ is a simple $\Lambda$-module. We decompose $\Lambda = P(S) \oplus Q$ where $P(S)$ is the projective cover of $S$. If $S$ satisfies
\begin{enumerate}
\item $\Ext^{i}_{\Lambda}(D \Lambda, S) = 0$ for any $0 \leq i < n$,

\item $\Ext^{i}_{\Lambda}(S, S) = 0$ for any $1 \leq i \leq n$,
\end{enumerate}
  then we call
$T = (\tau_{n}^{-}S) \oplus Q$
the $n$-$BB$ $tilting$ $module$ associated with $S$.
\end{definition}

Recall the generalized tilting modules \cite{BO4,M86}. Note that the tilting $\Lambda$-module $T$ with $\pd_{\Lambda} T \leq1$ is the classical tilting module \cite{ASS,H88}.
\begin{definition}
 Let $\Lambda$ be a finite dimensional algebra. An $\Lambda$-module $T \in \mod\Lambda$ is called tilting module with $\pd_{\Lambda} T \leq m$ if there exists $m \geq 0$ such that
\begin{enumerate}
\item $\pd_{\Lambda} T \leq m$,

\item $\Ext^{i}_{\Lambda}(T,T)=0$ for any $i > 0$,

\item there exists an exact sequence
$0 \xrightarrow{} \Lambda \xrightarrow{} T_{0} \xrightarrow{} \dots\ \xrightarrow{} T_{m} \xrightarrow{}0$
with $T_{i} \in \add T$.
\end{enumerate}
\end{definition}

Note that the $1$-APR, $1$-BB tilting module is just the classical APR, BB tilting module respectively. An $n$-APR tilting module is a  special $n$-BB tilting module associated with simple projective module.
Furthermore, an $n$-APR tilting module $T$ is in fact a tilting module with $\pd_{\Lambda} T = n$ (see \cite[Theorem 3.2] {IO12}) and an $n$-BB tilting module $T$ is a tilting module with $\pd_{\Lambda} T \leq n$ (see \cite[Lemma 4.2] {HX}).
The tensor products of tilting modules with finite projective dimension was studied in \cite{CC17,M86} and tensor products preserves tilting properties.
The purpose of this paper is to consider the tensor products of higher APR (BB) tilting modules by the corresponding simple modules.

\section{Preparation}\label{Prep}

In this section, in the setting of tensor products of finite dimensional algebras, our aim is to discuss tensor products of modules and complexes.
To apply tensor products without confusion, we use the symbol $\otimes$ for modules and $\otimes^{T}$ for complexes.

\subsection{Tensor products and semisimple, basic algebras}

We need the following two results.
\begin{proposition}\label{canonical-map}
Assume that $\Lambda,\Gamma$ are finite dimensional algebras over field $K$. Let $M,N \in \mod\Lambda$ and~$M^{'},N^{'} \in \mod\Gamma$. Then the canonical map
$$\Hom_{\Lambda}(M,N)\otimes \Hom_{\Gamma}(M^{'},N^{'}) \stackrel{}{\longrightarrow} \Hom_{\Lambda \otimes \Gamma}(M \otimes M^{'}, N \otimes N^{'})$$
given by $f \otimes g \stackrel{}{\longrightarrow} f \otimes g$ is an isomorphism of $K$-vector spaces.
\end{proposition}
\begin{proof}
It follows as an consequence of Proposition $\uppercase\expandafter{\romannumeral11}$.1.2.3 and Theorem $\uppercase\expandafter{\romannumeral11}$.3.1 in \cite{CE}.
\end{proof}

\begin{proposition}\label{prop-Tensor-product-of-exact-sequence}
Assume that $\Lambda,\Gamma$ are  finite dimensional algebras over field $K$. Let $0 \xrightarrow{} M \xrightarrow{f} N \xrightarrow{g} L \xrightarrow{} 0$ be an exact sequence in $\mod\Lambda$ and $0 \xrightarrow{} M^{'} \xrightarrow{f^{'}} N^{'} \xrightarrow{g^{'}} L^{'} \xrightarrow{} 0$ an exact sequence in $\mod\Gamma$. Then the following sequence
$$0 \xrightarrow{} (f \otimes 1)(M \otimes N^{'}) +  (1 \otimes f^{'})(N \otimes M^{'}) \xrightarrow{} N \otimes N^{'} \xrightarrow{g \otimes g^{'}} L \otimes L^{'} \xrightarrow{} 0$$
is an exact sequence in $\mod(\Lambda \otimes \Gamma)$.
\end{proposition}
\begin{proof}
Since, in the setting of tensor products over fields, the  tensor product bifunctor is an exact functor, the diagram
$$
\xymatrix@C=1.0cm@R1.0cm{
&         &0 \ar[d]                     &0   \ar[d]                     &0  \ar[d]                    &   \\
&0 \ar[r] &{M \otimes M^{'}} \ar[r]^{1 \otimes f^{'}}\ar[d]^{f \otimes 1}       &{M \otimes N^{'}}  \ar[r]^{1 \otimes g^{'}}\ar[d]^{f \otimes 1}
                                                            &{M \otimes N^{'}}   \ar[r]\ar[d]^{f \otimes 1}        &0  \\
&0 \ar[r] &{N \otimes M^{'}}  \ar[r]^{1 \otimes f^{'}}\ar[d]^{g \otimes 1}      &{N \otimes N^{'}}   \ar[r]^{1 \otimes g^{'}}\ar[d]^{g \otimes 1}
                                                            &{N \otimes L^{'}}   \ar[r]\ar[d]^{g \otimes 1}        &0\\
&0 \ar[r] &{L \otimes M^{'}}   \ar[r]^{1 \otimes f^{'}}\ar[d]  &{L \otimes N^{'}}  \ar[r]^{1 \otimes g^{'}}\ar[d]     &{L \otimes L^{'}}  \ar[r]\ar[d]     &0  \\
&         &0                            &0                              &0                            &
}$$
is commutative with columns and rows exact.
Because $(g \otimes g^{'})(f \otimes 1) =0 =(g \otimes g^{'})(1 \otimes f^{'})$, so $(f \otimes 1)(M \otimes N^{'}) +  (1 \otimes f^{'})(N \otimes M^{'}) \subseteq \Ker(g \otimes g^{'})$. Observed that morphism $g \otimes g^{'}$ is surjective.  It suffices to show that $\Ker(g \otimes g^{'}) \subseteq (f \otimes 1)(M \otimes N^{'}) +  (1 \otimes f^{'})(N \otimes M^{'})$.

 Assume $x \in \Ker(g \otimes g^{'})$, $0 =(g \otimes g^{'})(x) =(g \otimes 1)(1 \otimes g^{'})(x)$,
 so $(1 \otimes g^{'})(x) \in \Ker(g \otimes 1) =\Img(f \otimes 1)$, hence there exists $x_{1} \in M \otimes N^{'}$ such that $(1 \otimes g^{'})(x) =(f \otimes 1)(x_{1})$. Because $1 \otimes g^{'}$ is surjective, so there exists $x_{2} \in M \otimes N^{'}$ such that $x_{1} =(1 \otimes g^{'})(x_{2})$
\begin{align*}
((1 \otimes g^{'}))(x - (f \otimes 1)(x_{2})) &= (1 \otimes g^{'})(x) - (1 \otimes g^{'})(f \otimes 1)(x_{2})\\
                                              &= (1 \otimes g^{'})(x) - (f \otimes 1)(1 \otimes g^{'})(x_{2})\\
                                              &= (1 \otimes g^{'})(x) - (f \otimes 1)(x_{1})\\
                                              &=0
\end{align*}
Hence $x - (f \otimes 1)(x_{2}) \in \Ker(1 \otimes g^{'}) = \Img(1 \otimes f^{'})$,
this implies that there exists $x_{3} \in N \otimes M^{'}$ such that $x - (f \otimes 1)(x_{2}) = (1 \otimes f^{'})(x_{3})$, so we have $x = (f \otimes 1)(x_{2}) + (1 \otimes f^{'})(x_{3}) \in (f \otimes 1)(M \otimes N^{'}) +  (1 \otimes f^{'})(N \otimes M^{'})$. Therefore $\Ker(g \otimes g^{'}) \subseteq (f \otimes 1)(M \otimes N^{'}) +  (1 \otimes f^{'})(N \otimes M^{'})$ and finish the proof.
\end{proof}

If $\Lambda$ is a finite dimensional algebra, then the ${\rm radical}$ $\rad\Lambda$ of $\Lambda$ is the largest nilpotent two-sided ideal in $\Lambda$. We consider the ${\rm radical}$ of tensor product of finite dimensional algebras.
\begin{proposition}\label{prop-rad-of-tensor-product}
Assume that $\Lambda,\Gamma$ are finite dimensional algebras over field $K$.  Then $(\Lambda/\rad \Lambda) \otimes (\Gamma/\rad \Gamma)$ is a semisimple algebra if and only if
$\rad(\Lambda \otimes \Gamma) =\rad\Lambda \otimes \Gamma  + \Lambda \otimes \rad\Gamma$ as a ideal of $\Lambda \otimes \Gamma$.
\end{proposition}
\begin{proof}
Noted that two exact sequences $0 \xrightarrow{} \rad\Lambda \xrightarrow{} \Lambda \xrightarrow{} \Lambda/\rad \Lambda \xrightarrow{} 0$
and  $0 \xrightarrow{} \rad\Gamma \xrightarrow{} \Gamma \xrightarrow{} \Gamma/\rad \Gamma \xrightarrow{} 0$. By Proposition \ref{prop-Tensor-product-of-exact-sequence}, the following sequence is exact
\begin{gather}\label{eq-exact-of-rad}
0 \xrightarrow{} \rad\Lambda \otimes \Gamma  + \Lambda \otimes \rad\Gamma \xrightarrow{} \Lambda \otimes \Gamma \xrightarrow{\beta} (\Lambda/\rad \Lambda) \otimes (\Gamma/\rad \Gamma) \xrightarrow{} 0.
\end{gather}

"$\Longleftarrow$" Suppose $\rad(\Lambda \otimes \Gamma) =\rad\Lambda \otimes \Gamma  + \Lambda \otimes \rad\Gamma$ as a ideal of $\Lambda \otimes \Gamma$. By the exact sequence (\ref{eq-exact-of-rad}), we get
\begin{align*}
(\Lambda \otimes \Gamma)/ {\rad(\Lambda \otimes \Gamma)} =  (\Lambda \otimes \Gamma) / ( \rad\Lambda \otimes \Gamma+ \Lambda \otimes \rad\Gamma )
                                                        \cong (\Lambda/\rad \Lambda) \otimes (\Gamma/\rad \Gamma).
\end{align*}
Hence $(\Lambda/\rad \Lambda) \otimes (\Gamma/\rad \Gamma)$ is a semisimple algebra.

"$\Longrightarrow$" Assume $(\Lambda/\rad \Lambda) \otimes (\Gamma/\rad \Gamma)$ is a semisimple algebra, then $\rad((\Lambda/\rad \Lambda) \otimes (\Gamma/\rad \Gamma)) =0$. Because $\rad\Lambda$ is a nilpotent ideal,
 so $\rad\Lambda \otimes \Gamma$ is a nilpotent ideal of $\Lambda \otimes \Gamma$, hence $\rad\Lambda \otimes \Gamma \subseteq \rad(\Lambda \otimes \Gamma)$. Similarly, $\Lambda \otimes \rad\Gamma \subseteq \rad(\Lambda \otimes \Gamma)$ since $\rad\Gamma$ is a nilpotent ideal.
Therefore $\rad\Lambda \otimes \Gamma  + \Lambda \otimes \rad\Gamma   \subseteq  \rad(\Lambda \otimes \Gamma)$.
It is enough to show that $\rad(\Lambda \otimes \Gamma) \subseteq  \rad\Lambda \otimes \Gamma  + \Lambda \otimes \rad\Gamma$.

By the exact sequence (\ref{eq-exact-of-rad}), $\beta(\rad(\Lambda \otimes \Gamma)) \subseteq \rad((\Lambda/\rad \Lambda) \otimes (\Gamma/\rad \Gamma)) =0$,  then $\beta(\rad(\Lambda \otimes \Gamma)) =0$, so $\rad(\Lambda \otimes \Gamma) \subseteq  \rad\Lambda \otimes \Gamma  + \Lambda \otimes \rad\Gamma$ and finish the proof.
\end{proof}

\begin{lemma}\label{lemma-characterization-of-semisimple} \cite[Chapter \uppercase\expandafter{\romannumeral1}. Wedderburn-Artin theorem 3.4] {ASS}
Assume that $\Lambda$ is a finite dimensional algebra over algebraically closed field $K$.
Then $\Lambda$ is a semisimple algebra if and only if there exist positive integers $m_{1}$,$m_{2}$,$\dots$,$m_{s}$ and an algebra isomorphism
$$
   \Lambda \cong  \mathbb{M}_{m_{1}}(K)  \bigoplus  \mathbb{M}_{m_{2}}(K)  \bigoplus \dots  \bigoplus  \mathbb{M}_{m_{s}}(K)
$$
where $\mathbb{M}_{m_{i}}(K)$, $1 \leq i \leq s$ are matrix algebras consisting of all $m_{i} \times m_{i}$ square matrices over field $K$.
\end{lemma}

The following Proposition suggested that under the condition of algebraically closed field, tensor products preserve semisimple algebras.
\begin{proposition}\label{prop-tensor-product-preserve-semisimple}
Assume that $\Lambda,\Gamma$ are finite dimensional algebras over field $K$.  If $K$ is algebraically closed, then $\Lambda$, $\Gamma$ are semisimple algebras if and only if $\Lambda \otimes \Gamma$ is semisimple.
\end{proposition}
\begin{proof}
"$\Longrightarrow$"
Assume $\Lambda,\Gamma$ be finite dimensional semisimple algebras.
By Lemma \ref{lemma-characterization-of-semisimple}, there exist positive integers $m_{1}$,$m_{2}$,$\dots$,$m_{s}$
and $n_{1}$,$n_{2}$,$\dots$,$n_{t}$ such that
$
   \Lambda \cong   \bigoplus\limits_{i=1}^{s}   \mathbb{M}_{m_{i}}(K)
$ and
$
   \Gamma \cong  \bigoplus\limits_{j=1}^{t}  \mathbb{M}_{n_{j}}(K).
$
Hence we get
\begin{align*}
     \Lambda \otimes \Gamma
   \cong  ( \bigoplus_{i=1}^{s}   \mathbb{M}_{m_{i}}(K))   \bigotimes    (\bigoplus_{j=1}^{t}  \mathbb{M}_{n_{j}}(K))
   \cong \bigoplus_{i=1}^{s} \bigoplus_{j=1}^{t}   \mathbb{M}_{m_{i}}(K)  \bigotimes \mathbb{M}_{n_{j}}(K)
\end{align*}

By Lemma \ref{lemma-characterization-of-semisimple}, it is enough to show that $\mathbb{M}_{m}(K)  \bigotimes \mathbb{M}_{n}(K)  \cong  \mathbb{M}_{mn}(K)$ for any positive integers $m,n$.
Suppose a $K$-basis of $\mathbb{M}_{m}(K)$ is the set of matrices $e_{ij}, 1 \leq i, j \leq m$, where $e_{ij}$ has the coefficient 1 in the position $(i, j)$ and the coefficient $0$ elsewhere. Similarly, Suppose a $K$-basis of $\mathbb{M}_{n}(K)$ is the set of matrices $f_{kl}, 1 \leq k, l \leq n$ and a $K$-basis of $\mathbb{M}_{mn}(K)$ is the set of matrices $h_{vw}, 1 \leq v, w \leq mn$ .
Then $e_{ij} \otimes f_{kl},  1 \leq i, j \leq m, 1 \leq i, j \leq n$ is a $K$-basis of $\mathbb{M}_{m}(K)  \bigotimes \mathbb{M}_{n}(K)$.
Let $\theta(e_{ij} \otimes f_{kl}) = h_{i+(k-1)m, j+(l-1)m}$, it is easy check that $\theta: \mathbb{M}_{m}(K)  \bigotimes \mathbb{M}_{n}(K)  \xrightarrow{}  \mathbb{M}_{mn}(K)$ is an algebra isomorphism. Then we get the desired result.

"$\Longleftarrow$"
Suppose $\Lambda \otimes \Gamma$ is a semisimple algebra, then $\rad(\Lambda \otimes \Gamma) =0$. Because $(\Lambda/\rad \Lambda)$, $(\Gamma/\rad \Gamma)$ are semisimple, by the proof of "only if part", the algebra $(\Lambda/\rad \Lambda) \otimes (\Gamma/\rad \Gamma)$ is semisimple. Hence by Proposition \ref{prop-rad-of-tensor-product}, the radical $\rad(\Lambda \otimes \Gamma) =\rad\Lambda \otimes \Gamma  + \Lambda \otimes \rad\Gamma$.
If $\Lambda$ is not semisimple, then $\rad \Lambda \neq 0$, so $0 \neq \rad\Lambda \otimes \Gamma \subseteq \rad(\Lambda \otimes \Gamma)$, a contradiction. Thus $\Lambda$ is semisimple. By symmetry, $\Gamma$ is also a semisimple algebra.
\end{proof}

The following result is a directly consequence of Proposition \ref{prop-rad-of-tensor-product} and proposition \ref{prop-tensor-product-preserve-semisimple}.
\begin{corollary}\label{coro-semisimple-and-radical}
Assume that $\Lambda,\Gamma$ are finite dimensional algebras over field $K$.  If $K$ is algebraically closed, then
$(\Lambda/\rad \Lambda) \otimes (\Gamma/\rad \Gamma)$ is a semisimple algebra and $\rad(\Lambda \otimes \Gamma) =\rad\Lambda \otimes \Gamma  + \Lambda \otimes \rad\Gamma$ as a ideal of $\Lambda \otimes \Gamma$.
\end{corollary}

 Particularly, if we assume that $\Lambda,\Gamma$ are quotient algebras of path algebras over any field $K$ modulo some admissible ideals, then $(\Lambda/\rad \Lambda) \otimes (\Gamma/\rad \Gamma)$ is semisimple.
Next we need to discuss the tensor products of basic algebras.

\begin{lemma}\label{lem-basic-algebras}\cite[Chapter \uppercase\expandafter{\romannumeral1}. Proposition 6.2] {ASS}
Assume that $\Lambda$ is a finite dimensional algebra over algebraically closed field $K$ with radical $\rad\Lambda$.
Then $\Lambda$ is basic if and only if the algebra $\Lambda/\rad \Lambda$ is isomorphic to a direct sum $K^{\oplus n}$
of $n$ copies of $K$ for some integer $n$.
\end{lemma}

Now we show that the tensor product of basic algebras is also a basic algebra.

\begin{proposition}\label{prop-tensor-of-basic-algebras}
Assume that $\Lambda,\Gamma$ are two algebras over field $K$. If $K$ is algebraically closed,
then $\Lambda,\Gamma$ are two basic finite dimensional algebras if and only if the algebra $\Lambda \otimes\Gamma$ is a basic finite dimensional algebra.
\end{proposition}
\begin{proof}
Note that $\Lambda,\Gamma$ are two finite dimensional algebras if and only if $\Lambda \otimes\Gamma$ is a finite dimensional algebra.
For finite dimensional algebras $\Lambda,\Gamma$ over algebraically closed field $K$, by Lemma \ref{lemma-characterization-of-semisimple}, there exist positive integers $m_{1},m_{2},\dots,m_{s}$ and $n_{1},n_{2},\dots,n_{t}$ such that algebra isomorphisms $\Lambda/\rad \Lambda \cong  \bigoplus\limits_{i=1}^{s}   \mathbb{M}_{m_{i}}(K)$ and $\Gamma/\rad \Gamma \cong  \bigoplus\limits_{j=1}^{t}  \mathbb{M}_{n_{j}}(K)$.
It is follows from Proposition \ref{prop-rad-of-tensor-product} and Proposition \ref{prop-tensor-product-preserve-semisimple} that
\begin{align*}
(\Lambda \otimes \Gamma)/ {\rad(\Lambda \otimes \Gamma)} & \cong (\Lambda/\rad \Lambda) \bigotimes (\Gamma/\rad \Gamma) \\
              &\cong  ( \bigoplus_{i=1}^{s}   \mathbb{M}_{m_{i}}(K))   \bigotimes    (\bigoplus_{j=1}^{t}  \mathbb{M}_{n_{j}}(K))\\
              &\cong \bigoplus_{i=1}^{s} \bigoplus_{j=1}^{t}   \mathbb{M}_{m_{i}}(K)  \bigotimes \mathbb{M}_{n_{j}}(K)\\
              &\cong \bigoplus_{i=1}^{s} \bigoplus_{j=1}^{t} \mathbb{M}_{m_{i}n_{j}}(K)
\end{align*}
Since for two positive integers $m,n$, it is easily follows that $\mathbb{M}_{mn}(K) \cong K$ if and only if $m=n=1$.
Hence $\Lambda \otimes\Gamma$ is a basic finite dimensional algebra if and only if $m_{i}=n_{j}=1$ for any positive integers $m_{i},n_{j}$ if and only if $\Lambda,\Gamma$ are two basic finite dimensional algebras by Lemma \ref{lem-basic-algebras}, the proof is done.
\end{proof}

\subsection{Tensor products and semisimple modules, projective covers}

\begin{lemma}\label{lemma-characterization-of-indecomposable}\cite[Chapter \uppercase\expandafter{\romannumeral1}. Lemma 4.6 and Corollary 4.8] {ASS}
Assume that $\Lambda$ is a  finite dimensional algebra over algebraically closed field $K$. Let $M \in \mod\Lambda$. Then $M$ is indecomposable if and only if $\End_{\Lambda}(M)$ is a  local algebra if and only if $\End_{\Lambda}(M) / \rad({\rm End}_{\Lambda}(M)) \cong K$.
\end{lemma}

We have the following result related to the tensor products of indecomposable modules.

\begin{proposition}\label{prop-tensor-products-of-indecomposable}
Assume that $\Lambda,\Gamma$ are  finite dimensional algebras over algebraically closed field $K$. Let $M \in \mod\Lambda$, $N \in \mod\Gamma$.  Then $M$ and $N$ are indecomposable modules if and only if $M \otimes N$ is an indecomposable $(\Lambda \otimes \Gamma)$-module.
\end{proposition}
\begin{proof}
Suppose that  $M \otimes N$ is indecomposable. By Proposition \ref{canonical-map} and Lemma \ref{lemma-characterization-of-indecomposable},
$\End_{\Lambda}(M) \otimes \End_{\Gamma}(N) = \End_{\Lambda \otimes \Gamma}(M \otimes N)$  is a local algebra.
It is follows from \cite[Theorem 3] {L76} that $\End_{\Lambda}(M)$ and $\End_{\Gamma}(N)$ are local.
Hence by Proposition \ref{canonical-map}, $M$ and $N$ are indecomposable modules.
Conversely,
assume $M \in \mod\Lambda$, $N \in \mod\Gamma$ are indecomposable modules.
By Lemma \ref{lemma-characterization-of-indecomposable}, $(\End_{\Lambda}(M) / \rad({\rm End}_{\Lambda}(M))) \otimes (\End_{\Gamma}(N) / \rad({\rm End}_{\Gamma}(N))) \cong   K \otimes K  \cong K$ is a local algebra.
By \cite[Theorem 4] {L76}, $\End_{\Lambda}(M) \otimes \End_{\Gamma}(N)$ is a local algebra, so is $\End_{\Lambda \otimes \Gamma}(M \otimes N)$ by Proposition \ref{canonical-map}.
Hence $M \otimes N$ is an indecomposable $(\Lambda \otimes \Gamma)$-module by Lemma \ref{lemma-characterization-of-indecomposable}.
\end{proof}

We need to consider semisimple modules over tensor products of algebras.
\begin{proposition}\label{prop-tensor-product-of-semisimple}
Assume that $\Lambda,\Gamma$ are two finite dimensional algebras over algebraically closed field $K$. Let $M \in \mod{\Lambda}$ and $N \in \mod{\Gamma}$.
\begin{enumerate}
\item $\rad(M \otimes N)=\rad M \otimes N     +   M \otimes \rad N$ as a submodule of $M \otimes N$.

\item $M \otimes N$ is a semisimple $(\Lambda \otimes \Gamma)$-module if and only if $M, N$ are semisimple modules.
\end{enumerate}
\end{proposition}
\begin{proof}
(1)Observed that by Corollary \ref{coro-semisimple-and-radical}
 $\rad(\Lambda \otimes \Gamma) =\rad\Lambda \otimes \Gamma  + \Lambda \otimes \rad\Gamma$ as a ideal of $\Lambda \otimes \Gamma$.
Hence,
we get the radical $\rad(M \otimes N)$ of module $M \otimes N$,
\begin{align*}
\rad(M \otimes N) &=(\rad(\Lambda \otimes \Gamma) )(M \otimes N)  \\
                         &=(\rad\Lambda \otimes \Gamma  + \Lambda \otimes \rad\Gamma )(M \otimes N) \\
                       &=(\rad\Lambda \otimes \Gamma )(M \otimes N)  +  (\Lambda \otimes \rad\Gamma )(M \otimes N) \\
                       &= (\rad\Lambda) M \otimes \Gamma N     +   \Lambda M \otimes (\rad\Gamma) N   \\
                       &=\rad M \otimes N     +   M \otimes \rad N
\end{align*}

(2)"$\Longrightarrow$"
Assume $M \otimes N$ is a semisimple $(\Lambda \otimes \Gamma)$-module, then $\rad(M \otimes N) =0$.
If $M$ is not a semisimple module, then $\rad M \neq 0$, this implies $0 \neq \rad M \otimes N  \subseteq \rad(M \otimes N)$ by (1), a contradiction. Hence $M$ is a semisimple module. By symmetry, $N$ is semisimple.

"$\Longrightarrow$"
Assume $M, N$ are semisimple modules, then $\rad M = 0= \rad N$. Consequently, by (1), $\rad(M \otimes N) =0$, therefore $M \otimes N$ is a semisimple $(\Lambda \otimes \Gamma)$-module.
\end{proof}

By Proposition \ref{prop-tensor-products-of-indecomposable} and Proposition \ref{prop-tensor-product-of-semisimple}, we obtain the following characterization of simple modules over tensor products.
\begin{corollary}\label{coro-tensor-product-of-simple}
Assume that $\Lambda,\Gamma$ are two finite dimensional algebras over algebraically closed field $K$. Let $M \in \mod{\Lambda}$ and $N \in \mod{\Gamma}$. Then $M \otimes N$ is a simple $(\Lambda \otimes \Gamma)$-module if and only if $M, N$ are simple modules.
\end{corollary}

For finite dimensional algebra $\Lambda$, any indecomposable projective $\Lambda$-module is the form  $\Lambda e$ for some primitive  idempotent $e \in \Lambda$.
The following result gives a useful criterion for primitive orthogonal idempotents, indecomposable projective modules and indecomposable injective modules over tensor products.
\begin{proposition}\label{prop-tensor-product-of-primitive-idempotents}
Assume that $\Lambda,\Gamma$ are two finite dimensional algebras over algebraically closed field $K$ and $n,m \geq 1$. Assume that $e_{i}, 1 \leq i \leq n$ and $f_{j}, 1 \leq j \leq m$ are complete set of primitive orthogonal idempotents of $\Lambda,\Gamma$  respectively, then
\begin{enumerate}
\item $e_{i} \otimes f_{j}$, $1 \leq i \leq n$, $1 \leq j \leq m$ is a complete set of primitive orthogonal idempotents of the tensor product algebra $\Lambda \otimes \Gamma$.

\item $\Lambda e_{i} \otimes \Gamma f_{j}$, $1 \leq i \leq n$, $1 \leq j \leq m$ is a complete set of indecomposable projective $(\Lambda \otimes \Gamma)$-modules.

\item $D(e_{i} \Lambda) \otimes D(f_{j} \Gamma)$, $1 \leq i \leq n$, $1 \leq j \leq m$ is a complete set of indecomposable injective $(\Lambda \otimes \Gamma)$-modules.
\end{enumerate}
\end{proposition}

 Observed the form of indecomposable projective and injective modules, now we consider the tensor products of general projective and injective modules over tensor products of algebras.
\begin{proposition}\label{prop-tensor-product-of-proj.-inj.}
Assume that $\Lambda,\Gamma$ are two finite dimensional algebras over algebraically closed field $K$. Let $P_{\Lambda}, I_{\Lambda} \in \mod{\Lambda}$ and $P_{\Gamma}, I_{\Gamma} \in \mod{\Gamma}$. Then
\begin{enumerate}
\item $P_{\Lambda}$, $P_{\Gamma}$ are projective modules if and only if $P_{\Lambda} \otimes P_{\Gamma}$ is a projective $(\Lambda \otimes \Gamma)$-module.

\item $I_{\Lambda}$, $I_{\Gamma}$ are injective modules if and only if  $I_{\Lambda} \otimes I_{\Gamma}$ is a injective $(\Lambda \otimes \Gamma)$-module.
\end{enumerate}
\end{proposition}
\begin{proof}
We only prove (1), the (2) is similar.
For any modules $P_{\Lambda} \in \mod{\Lambda}$ and $P_{\Gamma} \in \mod{\Gamma}$, we decompose
$P_{\Lambda} = \bigoplus\limits_{i =1}^{s} M_{i},~ P_{\Gamma} = \bigoplus\limits_{j =1}^{t} N_{j}$
where $M_{i} \in \mod{\Lambda}$, $N_{j} \in \mod{\Gamma}$ are indecomposable. So
\begin{align}\label{eq-prop-tensor-product-of-proj.}
P_{\Lambda} \otimes P_{\Gamma} = (\bigoplus\limits_{i =1}^{s} M_{i}) \otimes (\bigoplus\limits_{j =1}^{t} N_{j}) \cong \bigoplus\limits_{i =1}^{s} \bigoplus\limits_{j =1}^{t} M_{i} \otimes N_{j}
\end{align}

"$\Longrightarrow$" Suppose $P_{\Lambda}$, $P_{\Gamma}$ are projective modules, then $M_{i} $, $N_{j}$ are indecomposable projective modules. By Proposition \ref{prop-tensor-product-of-primitive-idempotents}(2), $M_{i} \otimes N_{j}$, $1 \leq i \leq s$, $1 \leq j \leq t$ are indecomposable projective $(\Lambda \otimes \Gamma)$-modules, so $P_{\Lambda} \otimes P_{\Gamma}$ is projective by (\ref{eq-prop-tensor-product-of-proj.}).

"$\Longleftarrow$" Assume $P_{\Lambda} \otimes P_{\Gamma}$ is  projective. By (\ref{eq-prop-tensor-product-of-proj.}) and Proposition \ref{prop-tensor-products-of-indecomposable}, $M_{i} \otimes N_{j}$, $1 \leq i \leq s$, $1 \leq j \leq t$ are indecomposable projective $(\Lambda \otimes \Gamma)$-modules. Noted the form of the indecomposable projective $(\Lambda \otimes \Gamma)$-modules,  by Proposition \ref{prop-tensor-product-of-primitive-idempotents}(2), $M_{i} $, $N_{j}$ are indecomposable projective, therefore $P_{\Lambda}$, $P_{\Gamma}$ are projective.
\end{proof}

We have the following result as a corollary of Corollary \ref{coro-tensor-product-of-simple} and Proposition \ref{prop-tensor-product-of-proj.-inj.}.
\begin{corollary}\label{coro-simple-proj.-and-inj.-over-tensor-products}
Assume that $\Lambda,\Gamma$ are two finite dimensional algebras over algebraically closed field $K$. Let $P_{\Lambda}, I_{\Lambda} \in \mod{\Lambda}$ and $P_{\Gamma}, I_{\Gamma} \in \mod{\Gamma}$. Then
\begin{enumerate}
\item $P_{\Lambda} \otimes P_{\Gamma}$ is a simple projective $(\Lambda \otimes \Gamma)$-module if and only if $P_{\Lambda}$, $P_{\Gamma}$ are simple projective $\Lambda$-respectively, $\Gamma$-modules.

\item $I_{\Lambda} \otimes I_{\Gamma}$ is a simple injective $(\Lambda \otimes \Gamma)$-module if and only if $I_{\Lambda}$, $I_{\Gamma}$ are simple injective $\Lambda$-respectively, $\Gamma$-modules.
\end{enumerate}
\end{corollary}

We next consider the tensor products of projective cover, the following result shows that tensor products preserves projective covers of modules.
\begin{proposition}\label{prop-tensor-product-of-proj.-cover}
Assume that $\Lambda,\Gamma$ are two finite dimensional algebras over algebraically closed field $K$.
\begin{enumerate}
\item If $P_{M}$, $P_{N}$ are the projective cover of modules $M \in \mod{\Lambda}$, $N \in \mod{\Gamma}$ respectively, then $P_{M} \otimes P_{N}$ is the projective cover of $(\Lambda \otimes \Gamma)$-module $M \otimes N$.

\item If $I_{M}$, $I_{N}$ are the injective envelope of modules $M \in \mod{\Lambda}$, $N \in \mod{\Gamma}$ respectively, then $I_{M} \otimes I_{N}$ is the injective envelope of $(\Lambda \otimes \Gamma)$-module $M \otimes N$.
\end{enumerate}
\end{proposition}
\begin{proof}
We only prove (1), the proof of (2) is dual to (1).
For any $\Lambda$-module $M \in \mod{\Lambda}$ and $\Gamma$-module $N \in \mod{\Gamma}$, suppose $P_{M}$, $P_{N}$ are the projective cover of modules $M$, $N$ respectively,
then $P_{M}$ and $P_{N}$ are projective, hence by Proposition \ref{prop-tensor-product-of-proj.-inj.}(1), $P_{M} \otimes P_{N}$ is projective.
Because $P_{M}$ is the projective cover of $\Lambda$-module $M$, so $M/\rad{M} \cong P_{M}/\rad{P_{M}}$. Similarly, $N/\rad{N} \cong P_{N}/\rad{P_{N}}$.
Observed two exact sequences
$0 \xrightarrow{} \rad{P_{M}} \xrightarrow{} P_{M} \xrightarrow{g_{1}} M/\rad{M} \xrightarrow{} 0$ and
$0 \xrightarrow{} \rad{P_{N}} \xrightarrow{} P_{N} \xrightarrow{g_{2}} N/\rad{N} \xrightarrow{} 0$,
by Proposition \ref{prop-Tensor-product-of-exact-sequence}, the sequence
$$0 \xrightarrow{} \rad{P_{M}} \otimes P_{N} +  P_{M} \otimes \rad{P_{N}} \xrightarrow{} P_{M} \otimes P_{N} \xrightarrow{g_{1} \otimes g_{2}} (M/\rad{M}) \otimes (N/\rad{N})\xrightarrow{} 0$$
is exact.
By Proposition \ref{prop-tensor-product-of-semisimple}(1), we have modules isomorphism $(P_{M} \otimes P_{N})/(\rad(P_{M} \otimes P_{N})) \cong (M/\rad{M}) \otimes (N/\rad{N})$.
If modules isomorphism
$$(M\otimes N)/\rad(M\otimes N) \cong (M/\rad{M}) \otimes (N/\rad{N}),$$
 then it is clear from \cite[Chapter \uppercase\expandafter{\romannumeral1}. Corollary 5.9] {ASS} that $P_{M} \otimes P_{N}$ is the projective cover of $(\Lambda \otimes \Gamma)$-module $M \otimes N$.
It is suffices to prove modules isomorphism $(M\otimes N)/\rad(M\otimes N) \cong (M/\rad{M}) \otimes (N/\rad{N})$.
Note two exact sequences
$0 \xrightarrow{} \rad{M} \xrightarrow{} M \xrightarrow{} M/\rad{M} \xrightarrow{} 0$ and
$0 \xrightarrow{} \rad{N} \xrightarrow{} N \xrightarrow{} N/\rad{N} \xrightarrow{} 0$,
by Proposition \ref{prop-Tensor-product-of-exact-sequence}, the sequence
$$0 \xrightarrow{} \rad{M} \otimes N +  M \otimes \rad{N} \xrightarrow{} M \otimes N \xrightarrow{} M/\rad{M} \otimes N/\rad{N} \xrightarrow{} 0$$
is exact.
By Proposition \ref{prop-tensor-product-of-semisimple}(1), we have $(M\otimes N)/\rad(M\otimes N) \cong (M/\rad{M}) \otimes (N/\rad{N})$. The assertion follows.
\end{proof}

\subsection{Tensor products and complexes}
For complexes over algebras, the tensor product over fields can be used to construct tensor product of complexes \cite{CE,M67,R09}.

 \begin{definition}
 Assume that $\Lambda,\Gamma$ are two finite dimensional algebras over field $K$. Let objects $A_{\bullet} \in \C(\mod\Lambda)$ and~$B_{\bullet} \in \C(\mod\Gamma)$.
  \textbf{Tensor product of complexes} $A_{\bullet} \otimes^{T} B_{\bullet}$ of $A_{\bullet}$ and $B_{\bullet}$
   over $K$ is defined as the complex
$((A_{\bullet} \otimes^{T} B_{\bullet})_{p},d^{A_{\bullet} \otimes^{T} B_{\bullet}}_{p}) \in \mathcal{C}(\mod(\Lambda \otimes \Gamma))$ where
  $(A_{\bullet} \otimes^{T} B_{\bullet})_{p} =\bigoplus\limits_{j \in \mathbb{Z}}~A_{j} \otimes B_{p-j}$
  and the differential $d^{A_{\bullet} \otimes^{T} B_{\bullet}}_{p}$ given by
  \begin{gather}
d^{A_{\bullet} \otimes^{T} B_{\bullet}}_{p}(v\otimes w) =d^{A}_{j}(v)\otimes w  + (-1)^j v\otimes d^{B}_{p-j}(w),~\forall v\otimes w \in  A_{j} \otimes^{T} B_{p-j}. \notag
\end{gather}
for every $p \in \mathbb{Z}$.
\end{definition}

We need the K{\"u}nneth formula \cite{CE,M67,R09} which is a vital tool to compute homological group for tensor products of complexes. Since modules over fields is flat, we obtain the following K{\"u}nneth formula over a field.

\begin{lemma}\label{lem-homological-functorial-isomorphism}
Assume that $\Lambda,\Gamma$ are two finite dimensional algebras over field $K$. If $A_{\bullet} \in \C(\mod\Lambda)$ and~$B_{\bullet} \in \C(\mod\Gamma)$, then for every integer $p \in \mathbb{Z}$, there is a homological functorial isomorphism
      $${\rm H}_{p}(A_{\bullet} \otimes^{T} B_{\bullet}) \cong \bigoplus_{i+j=p} {\rm H}_{i}(A_{\bullet}) \otimes {\rm H}_{j}(B_{\bullet}).$$
\end{lemma}

Because tensor products over field preserve projective resolution \cite[\uppercase\expandafter{\romannumeral9}.Corollary 2.7] {CE}, by
Lemma \ref{lem-homological-functorial-isomorphism}, we have the following result.
\begin{lemma}\label{lem-Ext}(see \cite{CC17,P-2})
Assume that $\Lambda,\Gamma$ are two finite dimensional algebras over field $K$. If $M,N \in \mod\Lambda$ and~$M^{'},N^{'} \in \mod\Gamma$, then there is a functorial isomorphism
$$\Ext^{p}_{\Lambda \otimes \Gamma}(M \otimes M^{'}, N \otimes N^{'}) \cong \bigoplus_{i+j=p}\Ext^{i}_{\Lambda}(M, N) \otimes \Ext^{j}_{\Gamma}(M^{'}, N^{'})$$
for every integer $p \geq 0$.
\end{lemma}

The Lemma \ref{lem-homological-functorial-isomorphism} and Lemma \ref{lem-Ext} are useful formula and  play an important role in studying tensor product of modules and complexes.

In order to prove main result, we need the following result.
\begin{lemma}\label{lem-Ext-condition}
Assume that $\Lambda,\Gamma$ are two finite dimensional algebras over field $K$ and $n,m \geq 1$. Let objects $M \in \mod{\Lambda}$ and $N \in \mod{\Gamma}$.
\begin{enumerate}
\item If $\Ext^{i}_{\Lambda}(D \Lambda, M) = 0$ for any $0 \leq i < n$ and $\Ext^{j}_{\Gamma}(D \Gamma, N) = 0$ for any $0 \leq j < m$.
Then
$\Ext^{q}_{\Lambda \otimes\Gamma}(D (\Lambda \otimes\Gamma), M \otimes N) = 0$ for any $0 \leq q < n+m$.

\item If $\Ext^{i}_{\Lambda}(M, \Lambda) = 0$ for any $0 \leq i < n$ and $\Ext^{j}_{\Gamma}(N, \Gamma) = 0$ for any $0 \leq j < m$.
Then
$\Ext^{q}_{\Lambda \otimes\Gamma}(M \otimes N, \Lambda \otimes\Gamma) = 0$ for any $0 \leq q < n+m$.
\end{enumerate}
\end{lemma}
\begin{proof}
We only prove (1), the proof of (2) is dual to (1).
Suppose $\Ext^{i}_{\Lambda}(D \Lambda, M) = 0$ for any $0 \leq i < n$ and $\Ext^{j}_{\Gamma}(D \Gamma, N) = 0$ for any $0 \leq j < m$.
Firstly, for any integer $0 \leq q < n+m$, we prove that $\Ext^{i}_{\Lambda}(D\Lambda, M) \otimes \Ext^{j}_{\Gamma}(D\Gamma, N) =0$ for $0 \leq i,j \leq q$ with $i+j=q$.

When $0 \leq q \leq n-1$, then $0 \leq i \leq n-1$, so $\Ext^{i}_{\Lambda}(D\Lambda, M)=0$, hence $\Ext^{i}_{\Lambda}(D\Lambda, M) \otimes \Ext^{j}_{\Gamma}(D\Gamma, N) =0$ for $0 \leq i,j < n+m$ with $i+j=q$.

When $n \leq q \leq n+m-1$, it is suffices to show $\Ext^{i}_{\Lambda}(D\Lambda, M) \otimes \Ext^{j}_{\Gamma}(D\Gamma, N) =0$ for $n \leq i \leq q$.
In  the case of $n \leq i \leq q$, noticed that $0 \leq j= q-i \leq m-1$, so $\Ext^{j}_{\Gamma}(D\Gamma, N) =0$, the assertion follows.

By Lemma \ref{lem-Ext}, we obtain
\begin{align*}
\Ext^{q}_{\Lambda \otimes\Gamma}(D (\Lambda \otimes\Gamma), M \otimes N)
                         &= \Ext^{q}_{\Lambda \otimes\Gamma}( D\Lambda \otimes D\Gamma, M \otimes N)   \\
                         &=\bigoplus_{i+j=q}\Ext^{i}_{\Lambda}(D\Lambda, M) \otimes \Ext^{j}_{\Gamma}(D\Gamma, N) \\
                       &=0
\end{align*}
for any $0 \leq q < n+m$. Then the result follows.
\end{proof}

Similar to that tensor products over field preserve projective resolution \cite[\uppercase\expandafter{\romannumeral9}.Corollary 2.7] {CE},
the following statement implies that tensor products over field preserves  finite dimensions of injective, projective module.
\begin{lemma}\label{lem-proj.-and-inj.-dimension-over-tensor-product}
Assume that $\Lambda,\Gamma$ are two finite dimensional algebras over field $K$ and $n,m \geq 1$. Let objects $M \in \mod{\Lambda}$ and $N \in \mod{\Gamma}$.
\begin{enumerate}
\item If $\id_{\Lambda} M = n$ and $\id_{\Gamma} N = m$,
then
$\id_{\Lambda \otimes\Gamma} (M \otimes N) = n +m$.

\item If $\pd_{\Lambda} M = n$ and $\pd_{\Gamma} N = m$,
then
$\pd_{\Lambda \otimes\Gamma} (M \otimes N) = n +m$.
\end{enumerate}
\end{lemma}
\begin{proof}
We only prove (1), the proof of (2) is dual to (1).
Because ${\rm id}_{\Lambda} M = n$, so there exists object $M^{'} \in \mod{\Lambda}$  such that $\Ext^{n}_{\Lambda}(M^{'}, M) \neq 0$.
Similarly, there exists object $N^{'} \in \mod{\Gamma}$  such that $\Ext^{m}_{\Lambda}(N^{'}, N) \neq 0$.
So $\Ext^{n}_{\Lambda}(M^{'}, M) \otimes \Ext^{m}_{\Lambda}(N^{'}, N) \neq 0$, this implies
$$\Ext_{\Lambda \otimes\Gamma}^{n+m}(M^{'} \otimes N^{'}, M \otimes N)
                         =\bigoplus_{i+j=n+m}\Ext^{i}_{\Lambda}(M^{'}, M) \otimes \Ext^{j}_{\Gamma}(N^{'}, N) \neq 0$$
Hence $\id_{\Lambda \otimes\Gamma} (M \otimes N) \geq n+m$.
It is suffices to show $\id_{\Lambda \otimes\Gamma} (M \otimes N) \leq n+m$.

Put injective resolutions of $M \in \mod{\Lambda}$ and $N \in \mod{\Gamma}$ as
$$
I_{\bullet}: 0 \stackrel{}{\longrightarrow} M \stackrel{}{\longrightarrow} I_{0}\stackrel{}{\longrightarrow} I_{-1}\stackrel{}{\longrightarrow} \cdots\ \stackrel{}{\longrightarrow}  I_{-(n-1)} \stackrel{}{\longrightarrow} I_{-n} \stackrel{}{\longrightarrow} 0,
$$
$$
E_{\bullet}: 0 \stackrel{}{\longrightarrow} N \stackrel{}{\longrightarrow} E_{0}\stackrel{}{\longrightarrow} E_{-1}\stackrel{}{\longrightarrow} \cdots\ \stackrel{}{\longrightarrow}  E_{-(m-1)} \stackrel{}{\longrightarrow} E_{-m} \stackrel{}{\longrightarrow} 0
$$
where $I_{i} \in \mod{\Lambda}$, $E_{j} \in \mod{\Gamma}$ are injective modules.
We consider the delete complexes of $I_{\bullet}$ and $E_{\bullet}$ as follows
$$
I_{\bullet}^{M}: 0 \stackrel{}{\longrightarrow}  I_{0}\stackrel{}{\longrightarrow} I_{-1}\stackrel{}{\longrightarrow} \cdots\ \stackrel{}{\longrightarrow}  I_{-(n-1)} \stackrel{}{\longrightarrow} I_{-n} \stackrel{}{\longrightarrow} 0,
$$
$$
E_{\bullet}^{N}: 0 \stackrel{}{\longrightarrow}  E_{0}\stackrel{}{\longrightarrow} E_{-1}\stackrel{}{\longrightarrow} \cdots\ \stackrel{}{\longrightarrow}  E_{-(m-1)} \stackrel{}{\longrightarrow} E_{-m} \stackrel{}{\longrightarrow} 0.
$$
So it has homology
\begin{align*}
{\rm H}_{i} (I_{\bullet}^{M}) =
               \begin{cases} 0,& i\neq 0\\
                             M,& i= 0 \end{cases},\mbox{ and }
{\rm H}_{i} (E_{\bullet}^{N}) =
               \begin{cases} 0,& i\neq 0\\
                             N,& i= 0 \end{cases}.
\end{align*}
By Proposition \ref{prop-tensor-product-of-proj.-inj.}, the complex
$$
I_{\bullet}^{M}\otimes^{T} E_{\bullet}^{N}: 0 \stackrel{}{\longrightarrow}  I_{0} \otimes E_{0} \stackrel{}{\longrightarrow} (I_{\bullet}^{M}\otimes E_{\bullet}^{N})_{-1} \stackrel{}{\longrightarrow} \cdots\ \stackrel{}{\longrightarrow} (I_{\bullet}^{M}\otimes E_{\bullet}^{N})_{-(n+m)} \stackrel{}{\longrightarrow} 0,
$$
is an injective $(\Lambda \otimes \Gamma)$-module complex, we compute its homology by Lemma \ref{lem-homological-functorial-isomorphism},
\begin{align}
{\rm H}_{q} (I_{\bullet}^{M} \otimes^{T} E_{\bullet}^{N}) =\bigoplus_{i+j=q} {\rm H}_{i}(I_{\bullet}^{M}) \otimes {\rm H}_{j}(E_{\bullet}^{N})
                                                = \begin{cases} 0,& q\neq 0\\
                                                                {\rm H}_{0}(I_{\bullet}^{M}) \otimes {\rm H}_{0}(E_{\bullet}^{N}) = M \otimes N,& q= 0 \end{cases}.
\end{align}
Therefore the complex $I_{\bullet}^{M}\otimes^{T} E_{\bullet}^{N}$ is a delete injective resolutions of $M \otimes N$, this implies $\id_{\Lambda \otimes\Gamma} (M \otimes N) \leq n+m$. The proof is done.
\end{proof}

\section{Main results}\label{Tensor-products-of-higher-tilting-modules}

\subsection{Tensor products and n-Auslander-Reiten translations}
 In this subsection, for finite dimensional algebras of finite global dimensions, we study the tensor products of higher Auslander-Reiten translations and discuss whether tensor products preserves $\tau_{n}$-finite algebra.

Any finite dimensional algebra is a semi-primary  algebra which has been introduced and studied in \cite{Aus55}, we get the following result by \cite[Theorem 16] {Aus55} related to global dimensions of tensor product of finite dimensional algebras.
\begin{lemma}\label{lem-global-dimension-of-ten.-product}
Assume that $\Lambda,\Gamma$ are two finite dimensional algebras over field $K$ with radical $\rad\Lambda$, $\rad\Gamma$  respectively. If $(\Lambda/\rad \Lambda) \otimes (\Gamma/\rad \Gamma)$ is a semisimple algebra,
then the global dimension ${\rm gl.dim} (\Lambda \otimes \Gamma) = {\rm gl.dim} \Lambda + {\rm gl.dim} \Gamma$.
\end{lemma}

When the global dimensions of finite dimensional algebras is finite, we investigate the relationship between the higher Auslander-Reiten translations and tensor products.
\begin{proposition}\label{prop-global-dim.-of-tensor}
Assume that $\Lambda,\Gamma$ are two finite dimensional algebras over algebraically closed field $K$ and $n,m \geq 1$. If  global dimensions ${\rm gl.dim} \Lambda \leq n$ and ${\rm gl.dim} \Gamma \leq m$, then
\begin{enumerate}
\item The global dimension ${\rm gl.dim} (\Lambda \otimes \Gamma) \leq n +m$.

\item The  $(n +m)$-Auslander-Reiten translation
$$\tau_{n+m}(M \otimes N) =  \tau_{n}M  \otimes  \tau_{m}N,
  ~\tau_{n+m}^{-}(M \otimes N) = \tau_{n}^{-}M  \otimes  \tau_{m}^{-} .$$
for every objects $M \in \mod{\Lambda}$ and $N \in \mod{\Gamma}$.
\end{enumerate}
\end{proposition}
\begin{proof}
(1)It is directly from Lemma \ref{lem-global-dimension-of-ten.-product}, since by Corollary \ref{coro-semisimple-and-radical}, $(\Lambda/\rad \Lambda) \otimes (\Gamma/\rad \Gamma)$ is a semisimple algebra when field $K$ is algebraically closed.

(2) For any $0 \leq i,j \leq n+m$ with $i+j=n+m$, we first prove that $\Ext^{i}_{\Lambda}(M, \Lambda) \otimes \Ext^{j}_{\Gamma}(N, \Gamma) =0$ for $i \neq n$.  Without loss of generality, suppose that $0 \leq i < n$. Then $m+1 \leq j= n+m-i \leq n+m$, so $\Ext^{j}_{\Gamma}(N, \Gamma) =0$ since ${\rm gl.dim} \Gamma \leq m$, this implies $\Ext^{i}_{\Lambda}(M, \Lambda) \otimes \Ext^{j}_{\Gamma}(N, \Gamma) =0$.

By (1) and Lemma \ref{lem-Ext}, for every objects $M \in \mod{\Lambda}$ and $N \in \mod{\Gamma}$,
\begin{align*}
\tau_{n+m}(M \otimes N) &=D\Ext_{\Lambda \otimes\Gamma}^{n+m}(M \otimes N, \Lambda \otimes\Gamma)  \\
                         &=D(\bigoplus_{i+j=n+m}\Ext^{i}_{\Lambda}(M, \Lambda) \otimes \Ext^{j}_{\Gamma}(N, \Gamma)) \\
                       &=D(\Ext^{n}_{\Lambda}(M, \Lambda) \otimes \Ext^{m}_{\Gamma}(N, \Gamma))  \\
                       &=D\Ext^{n}_{\Lambda}(M, \Lambda) \otimes D\Ext^{m}_{\Gamma}(N, \Gamma)   \\
                       &=\tau_{n}M  \otimes  \tau_{m}N
\end{align*}
Hence the first assertion follows, the proof of the second assertion is  similar.
\end{proof}

As a application of the above Proposition, we discuss the relationship between tensor products and $\tau_{n}$-finite algebras.

\begin{theorem}\label{them-Tensor-products-and-tau{n}-finite}
Suppose that  $\Lambda,\Gamma$ are two finite dimensional algebras over algebraically closed field $K$ and $n,m \geq 1$. Assume global dimensions ${\rm gl.dim} \Lambda \leq n$ and ${\rm gl.dim} \Gamma \leq m$. Then
\begin{enumerate}
\item $\Lambda$ is $\tau_{n}$-finite or $\Gamma$ is $\tau_{m}$-finite if and only if $\Lambda \otimes \Gamma$ is $\tau_{(n+m)}$-finite.

\item $\Lambda$ is $\tau_{n}$-infinite and $\Gamma$ is $\tau_{m}$-infinite if and only if $\Lambda \otimes \Gamma$ is $\tau_{(n+m)}$-infinite.
\end{enumerate}
\end{theorem}
\begin{proof}
We only prove (1), the proof of (2) is similar to (1). Under the assumption of global dimensions ${\rm gl.dim} \Lambda \leq n$ and ${\rm gl.dim} \Gamma \leq m$,
by Proposition \ref{prop-global-dim.-of-tensor},
we obtain $\tau_{(n+m)}^{i}(D(\Lambda \otimes\Gamma)) =\tau_{(n+m)}^{i}(D\Lambda \otimes D\Gamma) =\tau_{n}^{i}(D\Lambda) \otimes \tau_{m}^{i}(D\Gamma)$ for $i \geq 0$.
Hence for positive integer $i_{0}$, it is follows that $\tau_{(n+m)}^{i_{0}}(D(\Lambda \otimes\Gamma))=0$ if and only if $\tau_{n}^{i_{0}}(D\Lambda)=0$ or $\tau_{m}^{i_{0}}(D\Gamma)=0$.
Therefore, $\Lambda$ is $\tau_{n}$-finite or $\Gamma$ is $\tau_{m}$-finite if and only if $\Lambda \otimes \Gamma$ is $\tau_{(n+m)}$-finite.
\end{proof}

As a corollary of the Theorem \ref{them-Tensor-products-and-tau{n}-finite}, we get the following results.

\begin{corollary}
Suppose that  $\Lambda,\Gamma$ are two finite dimensional algebras over algebraically closed field $K$ and $n,m \geq 1$.
\begin{enumerate}
\item If $\Lambda$ is $\tau_{n}$-finite, then ${\rm gl.dim} \Gamma \leq m$ if and only if $\Lambda \otimes \Gamma$ is $\tau_{(n+m)}$-finite.

\item If $\Lambda$ is $\tau_{n}$-infinite, then $\Gamma$ is $\tau_{m}$-infinite if and only if $\Lambda \otimes \Gamma$ is $\tau_{(n+m)}$-infinite.
\end{enumerate}
\end{corollary}
\begin{proof}
We only prove (1), the proof of (2) is similar to (1).
Assume $\Lambda$ is $\tau_{n}$-finite, then ${\rm gl.dim} \Lambda \leq n$.
Hence by Lemma \ref{lem-global-dimension-of-ten.-product}, ${\rm gl.dim} \Gamma \leq m$ if and only if ${\rm gl.dim} (\Lambda \otimes \Gamma)\leq n+m$.
The rest is obtained by Theorem \ref{them-Tensor-products-and-tau{n}-finite}(1).
\end{proof}

\begin{corollary}
Suppose that $\Lambda$ is a finite dimensional algebra over algebraically closed field $K$ and $n\geq 1$. Let $\Lambda^{e} =\Lambda \otimes_{K} \Lambda^{op}$ be the enveloping algebra of $\Lambda$.
Then
 $\Lambda$ is $\tau_{n}$-finite if and only if $\Lambda^{e}$ is $\tau_{2n}$-finite.
\end{corollary}
\begin{proof}
Since ${\rm gl.dim} \Lambda ={\rm gl.dim} \Lambda^{op}$, so by Lemma \ref{lem-global-dimension-of-ten.-product},
${\rm gl.dim} \Lambda^{e} ={\rm gl.dim} \Lambda +{\rm gl.dim} \Lambda^{op}$,
hence ${\rm gl.dim} \Lambda \leq n$ if and only if ${\rm gl.dim} \Lambda^{e} \leq 2n$.
By Lemma \ref{lem-tau{n}-finite} and Theorem \ref{them-Tensor-products-and-tau{n}-finite},
$\Lambda$ is $\tau_{n}$-finite if and only if $\Lambda^{op}$ is $\tau_{n}$-finite if and only if $\Lambda^{e}$ is $\tau_{2n}$-finite.
\end{proof}

\begin{proposition}\label{prop-tensor-product-333}
Suppose that $K$ is an algebraically closed field and $n,m \geq 1$. Assume $\Lambda$ is an $n$-representation finite algebra and $\Gamma$ is an $m$-representation infinite algebra. Then
 $\Lambda \otimes \Gamma$ is $\tau_{(n+m)}$-finite and neither $(n+m)$-representation infinite nor $(n+m)$-representation finite.
\end{proposition}
\begin{proof}
Because $\Lambda$ is a special $n$-complete algebra which is $\tau_{n}$-finite, so by Theorem \ref{them-Tensor-products-and-tau{n}-finite} $\Lambda \otimes \Gamma$ is $\tau_{(n+m)}$-finite, this implies that $\Lambda \otimes \Gamma$ is not $(n+m)$-representation infinite.
Assume that $I_{\Lambda} \in \mod\Lambda$ and $I_{\Gamma} \in \mod\Gamma$ are indecomposable injective modules,
then by Proposition \ref{prop-tensor-product-of-primitive-idempotents} $I_{\Lambda} \otimes I_{\Gamma}$ is a indecomposable injective $(\Lambda \otimes\Gamma)$-module. By \cite[Proposition 4.21] {HIO14}, $\tau_{m}^{i}(I_{\Gamma})=\nu_{m}^{i}(I_{\Gamma}) \neq 0$ for $i \geq 0$, so $\tau_{m}^{i}(I_{\Gamma})$ is not projective for $i \geq 0$.
This implies $\tau_{(n+m)}^{i}(I_{\Lambda} \otimes I_{\Gamma}) =\tau_{n}^{i}(I_{\Lambda}) \otimes \tau_{m}^{i}(I_{\Gamma})$ is not projective for $i \geq 0$ by Proposition \ref{prop-tensor-product-of-primitive-idempotents}.
It is follows from \cite[Proposition 1.3(b)] {I11} that $\Lambda \otimes \Gamma$ is not $(n+m)$-representation finite, we complete the proof.
\end{proof}

\subsection{Tensor products of higher APR tilting modules}
Higher APR tilting modules and higher BB tilting modules was introduced and studied in higher Auslander-Reiten theory.
In this subsection, we study how to construct new higher APR tilting modules and higher BB tilting modules over tensor products of algebras.

Noticed tensor product of basic finite dimensional algebras is also a basic finite dimensional algebra. Firstly we construct higher BB tilting modules by tensor products.

\begin{theorem}\label{prop-n-BB-and-finite-global-dimensions}
Suppose that  $\Lambda,\Gamma$ are basic finite dimensional algebras  over algebraically closed field $K$ and $n,m \geq 1$. Assume global dimensions ${\rm gl.dim} \Lambda \leq n$ and ${\rm gl.dim} \Gamma \leq m$. Suppose that $S_{\Lambda} \in \mod\Lambda$, $S_{\Gamma} \in \mod\Gamma$ are simple modules.
Let $P_{\Lambda} \in \mod\Lambda$, $P_{\Gamma} \in \mod\Gamma$ be the projective cover of $S_{\Lambda}$, $S_{\Gamma}$ respectively.
If
$$T_{\Lambda} = (\tau_{n}^{-}S_{\Lambda}) \oplus ({\Lambda} / P_{\Lambda}), ~T_{\Gamma} = (\tau_{m}^{-}S_{\Gamma}) \oplus ({\Gamma} / P_{\Gamma})$$
are $n$-respectively $m$-$BB$ tilting modules,
then  $T_{\Lambda \otimes \Gamma} = \tau_{n+m}^{-}(S_{\Lambda}\otimes S_{\Gamma}) \oplus ((\Lambda \otimes \Gamma) / (P_{\Lambda} \otimes P_{\Gamma}))$ is an $(n+m)$-BB tilting $(\Lambda \otimes \Gamma)$-module associated with $S_{\Lambda}\otimes S_{\Gamma}$.
\end{theorem}
\begin{proof}
Because $P_{\Lambda} \in \mod\Lambda$, $P_{\Gamma} \in \mod\Gamma$ are the projective cover of the simple modules $S_{\Lambda}$, $S_{\Gamma}$ respectively,
by Corollary \ref{coro-tensor-product-of-simple} and  Proposition \ref{prop-tensor-product-of-proj.-cover}, $P_{\Lambda} \otimes P_{\Gamma}$ is the projective cover of the simple $(\Lambda \otimes \Gamma)$-module $S_{\Lambda}\otimes S_{\Gamma}$.

By Definition \ref{def-n-BB-tilting-module}(1),
we get that $\Ext^{i}_{\Lambda}(D \Lambda, S_{\Lambda}) = 0$ for any $0 \leq i < n$ and $\Ext^{j}_{\Gamma}(D \Gamma, S_{\Gamma}) = 0$ for any $0 \leq j < m$.
Hence by Lemma \ref{lem-Ext-condition}, we have $\Ext^{q}_{\Lambda \otimes\Gamma}(D (\Lambda \otimes\Gamma), S_{\Lambda}\otimes S_{\Gamma}) = 0$ for any $0 \leq q < n+m$.
It is suffices to show $\Ext^{i}_{\Lambda \otimes\Gamma}(S_{\Lambda}\otimes S_{\Gamma}, S_{\Lambda}\otimes S_{\Gamma}) = 0$ for any $1 \leq i \leq n+m$.

By Definition \ref{def-n-BB-tilting-module}(2),
$\Ext^{i}_{\Lambda}(S_{\Lambda}, S_{\Lambda}) = 0$ for any $1 \leq i \leq n$ and $\Ext^{j}_{\Lambda}(S_{\Gamma}, S_{\Gamma}) = 0$ for any $1 \leq j \leq m$.
Under the condition ${\rm gl.dim} \Lambda \leq n$ and ${\rm gl.dim} \Gamma \leq m$, it is follows that $\Ext^{i}_{\Lambda}(S_{\Lambda}, S_{\Lambda}) = 0$ and $\Ext^{i}_{\Lambda}(S_{\Gamma}, S_{\Gamma}) = 0$ for any $1 \leq i$.
This implies $\Ext^{i}(S_{\Lambda}, S_{\Lambda}) \otimes \Ext^{j}(S_{\Gamma}, S_{\Gamma}) =0$ for $i>0$ or $j>0$. Thus
$
      \Ext^{q}(S_{\Lambda}\otimes S_{\Gamma}, S_{\Lambda}\otimes S_{\Gamma})
     = \bigoplus\limits_{i+j=q} \Ext^{i}(S_{\Lambda}, S_{\Lambda}) \otimes \Ext^{j}(S_{\Gamma}, S_{\Gamma}) =0
$
for $1 \leq q \leq n+m$. The proof is done.
\end{proof}

On above Theorem, in the setting of global dimensions ${\rm gl.dim} \Lambda \leq n$ and ${\rm gl.dim} \Gamma \leq m$, by Proposition \ref{prop-global-dim.-of-tensor}, $\tau_{n+m}^{-}(S_{\Lambda}\otimes S_{\Gamma})=\tau_{n}^{-}S_{\Lambda} \otimes \tau_{m}^{-}S_{\Gamma}$.
Moreover, the condition ${\rm gl.dim} \Lambda \leq n$ and ${\rm gl.dim} \Gamma \leq m$ is not necessary.
In fact, it is enough to assume that $\Ext^{i}_{\Lambda}(S_{\Lambda}, S_{\Lambda}) = 0$ and $\Ext^{i}_{\Lambda}(S_{\Gamma}, S_{\Gamma}) = 0$ for any $1 \leq i \leq n+m$,
note that this assumption is automatic if we consider the higher BB tilting modules associated with simple projective modules $S_{\Lambda}$ and $S_{\Gamma}$ which is just the weak higher APR tilting modules associated with $S_{\Lambda}$ and $S_{\Gamma}$.
Now in general we construct higher APR tilting modules by tensor products.

\begin{theorem}\label{the-tensor-products-of-n-APR-tilting-modules}
Suppose that  $\Lambda,\Gamma$ are basic finite dimensional algebras  over algebraically closed field $K$ and $n,m \geq 1$.
Let $P_{\Lambda}, P_{\Gamma}$ be simple projective $\Lambda$-respectively, $\Gamma$-modules.
Let $T_{\Lambda \otimes \Gamma} = (\tau_{n+m}^{-}(P_{\Lambda} \otimes P_{\Gamma})) \oplus ((\Lambda \otimes \Gamma) / (P_{\Lambda} \otimes P_{\Gamma}))$.
If
$$T_{\Lambda} = (\tau_{n}^{-}P_{\Lambda}) \oplus ({\Lambda} / P_{\Lambda}), ~T_{\Gamma} = (\tau_{m}^{-}P_{\Gamma}) \oplus ({\Gamma} / P_{\Gamma})$$
are weak $n$-respectively $m$-APR tilting modules, then
\begin{enumerate}
\item $T_{\Lambda \otimes \Gamma}$ is a weak $(n+m)$-APR tilting module associated with $P_{\Lambda} \otimes P_{\Gamma}$.

\item If moreover $\id P_{\Lambda} = n$ and $\id P_{\Gamma} = m$, then $T_{\Lambda \otimes \Gamma}$ is an $(n+m)$-APR tilting module.

\item If global dimensions ${\rm gl.dim} \Lambda = n$ and ${\rm gl.dim} \Gamma = m$, then the global dimension ${\rm gl.dim} \Omega =n+m$ where the $(n+m)$-APR tilt algebra $\Omega=\End_{\Lambda \otimes \Gamma}(T_{\Lambda \otimes \Gamma})^{op}$.
\end{enumerate}
\end{theorem}
\begin{proof}
(1)Because $P_{\Lambda}, P_{\Gamma}$ are simple projective modules, so $P_{\Lambda} \otimes P_{\Gamma}$ is a simple projective module by Corollary \ref{coro-simple-proj.-and-inj.-over-tensor-products}. Since $T_{\Lambda}$, $T_{\Gamma}$ are weak $n$-respectively $m$-APR tilting modules, we get that $\Ext^{i}_{\Lambda}(D \Lambda, P_{\Lambda}) = 0$ for any $0 \leq i < n$ and $\Ext^{j}_{\Gamma}(D \Gamma, P_{\Gamma}) = 0$ for any $0 \leq j < m$. It is follows from Lemma \ref{lem-Ext-condition} that
$\Ext^{q}_{\Lambda \otimes\Gamma}(D (\Lambda \otimes\Gamma), P_{\Lambda} \otimes P_{\Gamma}) = 0$ for any $0 \leq q < n+m$. Consequently, the assertion follows.

(2)By Lemma \ref{lem-proj.-and-inj.-dimension-over-tensor-product}, we get $\id_{\Lambda \otimes\Gamma} (P_{\Lambda} \otimes P_{\Gamma}) = n +m$ by assumption. The rest is directly obtained from (1).

(3)By Lemma \ref{lem-global-dimension-of-ten.-product}, the global dimension ${\rm gl.dim} (\Lambda \otimes \Gamma) =n+m$. It is follows from (1) and \cite[Proposition 3.6] {IO12} that the global dimension ${\rm gl.dim} \Omega =n+m$ where the algebra $\Omega=\End_{\Lambda \otimes \Gamma}(T_{\Lambda \otimes \Gamma})^{op}$.
\end{proof}

Theorem \ref{the-tensor-products-of-n-APR-tilting-modules} proves that the $(n+m)$-APR tilting module over tensor products must exist if there exist $n$-respectively, $m$-APR tilting module over original algebras.
Moreover, Theorem \ref{prop-n-BB-and-finite-global-dimensions} and Theorem \ref{the-tensor-products-of-n-APR-tilting-modules} is also the construction of tilting modules with $\pd_{\Lambda} T \leq n+m$.
The following result related to  higher APR cotilting module is dual to Theorem \ref{the-tensor-products-of-n-APR-tilting-modules}.

\begin{theorem}\label{the-tensor-products-of-n-APR-cotilting-modules}
Suppose that  $\Lambda,\Gamma$ are basic finite dimensional algebras  over algebraically closed field $K$ and $n,m \geq 1$.
Let $I_{\Lambda}, I_{\Gamma}$ are simple injective $\Lambda$-respectively, $\Gamma$-modules.
Let $T_{\Lambda \otimes \Gamma} = (\tau_{n+m}(I_{\Lambda} \otimes I_{\Gamma})) \oplus ((\Lambda \otimes \Gamma)/(I_{\Lambda} \otimes I_{\Gamma}))$.
If
$$T_{\Lambda} = (\tau_{n}I_{\Lambda}) \oplus (D\Lambda/I_{\Lambda}), ~T_{\Gamma} = (\tau_{m}I_{\Gamma}) \oplus (D\Gamma/I_{\Gamma})$$
are weak $n$-respectively $m$-$APR$ cotilting modules, then
\begin{enumerate}
\item $T_{\Lambda \otimes \Gamma}$ is a weak $(n+m)$-$APR$ cotilting module associated with $I_{\Lambda} \otimes I_{\Gamma}$.

\item If moreover $\pd I_{\Lambda} = n$ and $\pd I_{\Gamma} = m$, then $T_{\Lambda \otimes \Gamma}$ is an $(n+m)$-$APR$ cotilting module.
\end{enumerate}
\end{theorem}

%%%%%

\subsection{Description-of-higher-APR-tilting-modules}
$n$-hereditary algebras as the generalization of hereditary algebras were introduced in higher representation theory. The following result is a characterization of $n$-hereditary algebras.

\begin{proposition}\label{prop-$n$-hereditary-algebra}\cite[Theorem 3.4] {HIO14}
Let $\Lambda$ be a ring-indecomposable finite dimensional algebra. Then $\Lambda$ is an $n$-hereditary algebra if and only if it is either $n$-representation finite or $n$-representation infinite.
\end{proposition}

Under certain conditions, tensor products preserves $n$-representation finiteness \cite{IO12} and $n$-representation infiniteness \cite{HIO14,MY16}.
Then it is natural to ask whether the tensor product $\Lambda \otimes \Gamma$ of $n$-hereditary algebra $\Lambda$ with $m$-hereditary algebra $\Gamma$ is $(n+m)$-hereditary.
Since tensor product of basic ring-indecomposable algebras is ring-indecomposable, Proposition \ref{prop-tensor-product-333} means the fact that tensor products does not preserve the property of $n$-hereditary in general.

Now we discuss the higher APR tilting modules over the tensor products of higher hereditary algebras and give the following description.

\begin{corollary}\label{coro-APR-over-tensor-products-of-higher-hereditary}
Suppose that $\Lambda,\Gamma$ are basic ring-indecomposable $n$-respectively $m$-hereditary algebras over algebraically closed field $K$ with positive integers $n,m \geq 1$.
Let $P_{\Lambda}, P_{\Gamma}$ be indecomposable projective and non-injective $\Lambda$-respectively, $\Gamma$-modules.
Let $T_{\Lambda} = (\tau_{n}^{-}P_{\Lambda}) \oplus ({\Lambda} / P_{\Lambda})$, $T_{\Gamma} = (\tau_{m}^{-}P_{\Gamma}) \oplus ({\Gamma} / P_{\Gamma})$ and $T_{\Lambda \otimes \Gamma} = (\tau_{n}^{-}P_{\Lambda} \otimes \tau_{m}^{-}P_{\Gamma}) \oplus ((\Lambda \otimes \Gamma) / (P_{\Lambda} \otimes P_{\Gamma}))$.
\begin{enumerate}
\item $T_{\Lambda}$ is an $n$-APR tilting $\Lambda$-module if and only if $P_{\Lambda}$ is a simple projective and non-injective $\Lambda$-module.

\item $T_{\Lambda}, T_{\Gamma}$ are $n$-respectively $m$-APR tilting modules if and only if $T_{\Lambda \otimes \Gamma}$ is an $(n+m)$-APR tilting $(\Lambda \otimes \Gamma)$-module.

\item If $\Lambda,\Gamma$ are non-semisimple, then
$$\lvert {\rm APR}_{\Lambda \otimes \Gamma} \rvert =\lvert {\rm APR}_{\Lambda} \rvert \lvert {\rm APR}_{\Gamma} \rvert,$$
here $\lvert {\rm APR}_{\Lambda} \rvert$, $\lvert {\rm APR}_{\Gamma} \rvert$, $\lvert {\rm APR}_{\Lambda \otimes \Gamma} \rvert$ are the numbers of the $n$-, $m$-, $(n+m)$-APR tilting $\Lambda$-, $\Gamma$-, $(\Lambda \otimes \Gamma)$-module which are obtained by different simple projective modules, respectively.

\item If $\Lambda,\Gamma$ are $l$-homogeneous $n$-respectively, $m$-representation finite for common integer $l$ and $P_{\Lambda}, P_{\Gamma}$ are simple projective and non-injective modules, then the $(n+m)$-APR tilt $\End_{\Lambda \otimes \Gamma}(T_{\Lambda \otimes \Gamma})^{op}$  is $(n+m)$-representation finite.

\item If $\Lambda,\Gamma$ are $n$-respectively $m$-representation infinite and $P_{\Lambda}, P_{\Gamma}$ are simple projective modules, then the $(n+m)$-APR tilt $\End_{\Lambda \otimes \Gamma}(T_{\Lambda \otimes \Gamma})^{op}$ is $(n+m)$-representation-infinite.
\end{enumerate}
\end{corollary}
\begin{proof}
(1)By Proposition \ref{prop-$n$-hereditary-algebra}, $\Lambda$ is either $n$-representation finite or $n$-representation infinite.
When $\Lambda$ is $n$-representation finite, by \cite[Observation 4.1] {IO12}, any simple projective and non-injective $\Lambda$-modules $P_{\Lambda}$ admits the $n$-APR tilting module associated with $P_{\Lambda}$.
When $\Lambda$ is $n$-representation infinite, by \cite[Section 2.2] {MY16}, any simple projective $\Lambda$-module $P_{\Lambda}$ gives an $n$-APR tilting $\Lambda$-module.
Hence by Definition \ref{def-n-APR-tilting-module} the assertion follows.

(2)Assume $T_{\Lambda}, T_{\Gamma}$ are $n$-respectively $m$-APR tilting modules, by Proposition \ref{prop-global-dim.-of-tensor} and Theorem \ref{the-tensor-products-of-n-APR-tilting-modules}, $T_{\Lambda \otimes \Gamma}$ is an $(n+m)$-APR tilting $(\Lambda \otimes \Gamma)$-module.
Conversely,
observed that by Proposition \ref{prop-tensor-product-of-primitive-idempotents} and Corollary \ref{coro-simple-proj.-and-inj.-over-tensor-products}, $P_{\Lambda} \otimes P_{\Gamma}$ is a simple projective and non-injective $(\Lambda \otimes \Gamma)$-module if and only if $P_{\Lambda}$, $P_{\Gamma}$ are simple projective and non-injective $\Lambda$-respectively, $\Gamma$-modules.
Thus if $T_{\Lambda \otimes \Gamma}$ is an $(n+m)$-APR tilting $(\Lambda \otimes \Gamma)$-module,
then $P_{\Lambda} \otimes P_{\Gamma}$ is a simple projective and non-injective $(\Lambda \otimes \Gamma)$-module, this implies by (1),
$T_{\Lambda}$ and $T_{\Gamma}$ are $n$-respectively $m$-APR tilting modules.

(3)When $\Lambda,\Gamma$ are basic ring-indecomposable and non-semisimple, any simple projective $\Lambda$-, $\Gamma$-modules are non-injective. The rest is obtained from (1) and (2).

(4) When $\Lambda,\Gamma$ are $l$-homogeneous $n$-respectively $m$-representation finite for common integer $l$, by \cite[Corollary 1.5] {HI11}, $\Lambda \otimes \Gamma$ is an $(n+m)$-representation finite algebra. Hence by (1),(2) and \cite[Corollary 4.3] {IO12}, $\End_{\Lambda \otimes \Gamma}(T_{\Lambda \otimes \Gamma})^{op}$  is $(n+m)$-representation finite.

(5) It is follows from \cite[Theorem 2.10] {HIO14} that $(\Lambda \otimes \Gamma)$ is $(n+m)$-representation infinite. By (1),(2) and \cite[Theorem 2.13] {HIO14}, $\End_{\Lambda \otimes \Gamma}(T_{\Lambda \otimes \Gamma})^{op}$ is $(n+m)$-representation infinite.
\end{proof}

Now we give an example to illustrate our results.
\begin{example}
Assume path algebra $\Lambda =KQ$ where the quiver
$$
Q:
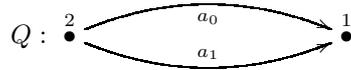
\xymatrix@C=1.5cm@R1.5cm{
\stackrel{2}{\bullet} \ar@/^1pc/[rr]_{a_{0}}  \ar@/_1pc/[rr]^{a_{1}} & &\stackrel{1}{\bullet}
}
$$
This is a Beilinson algebra of dimension 1 and $1$-representation infinite algebra by \cite{HIO14}. We study the tensor product algebra $\Gamma=\Lambda \otimes \Lambda$ which is $2$-representation infinite and $\tau_{2}$-infinite.
Let $e_{i}$ is the trivial path corresponding to vertices $i \in \{1,2\}$, then $e_{i,j} =e_{i}\otimes e_{j}$, $i,j \in \{1,2\}$ is a complete set of primitive orthogonal idempotents of
the algebra $\Gamma$.
By \cite{Le94}, the algebra $\Gamma$ is defined by the quiver
$$
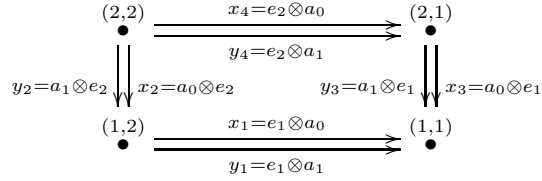
\xymatrix@C=1.5cm@R0.8cm{
   &{\stackrel{(2,2)}{\bullet}} \ar@{->}@<2pt>[rr]^{x_{4}= e_{2}\otimes a_{0}} \ar@{->}@<-2pt>[rr]_{y_{4}= e_{2}\otimes a_{1}}  \ar@{->}@<2pt>[d]^{x_{2}= a_{0}\otimes e_{2}} \ar@{->}@<-2pt>[d]_{y_{2}= a_{1}\otimes e_{2}}
   &      &{\stackrel{(2,1)}{\bullet}} \ar@{->}@<2pt>[d]^{x_{3}= a_{0}\otimes e_{1}} \ar@{->}@<-2pt>[d]_{y_{3}= a_{1}\otimes e_{1}}   \\
&{\stackrel{(1,2)}{\bullet}}  \ar@{->}@<2pt>[rr]^{x_{1}= e_{1}\otimes a_{0}} \ar@{->}@<-2pt>[rr]_{y_{1}= e_{1}\otimes a_{1}}
   &      &{\stackrel{(1,1)}{\bullet}}
}
$$
with relations $(a_{i}\otimes e_{1})(e_{2}\otimes a_{j})=(e_{1}\otimes a_{j})(a_{i}\otimes e_{2})$, $i,j \in \{0,1\}$.

Let $P_{i+2l}=\tau^{-l}_{1}(\Lambda e_{i})$ for vertices $i \in \{1,2\}$ and $l\geq 0$,  the quiver of the category $\add\{P_{j}| j\geq 0\}$ is the following
$$
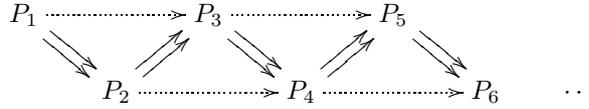
\xymatrix@C=0.6cm@R0.5cm{
&{P_{1}} \ar@{->}@<2pt>[dr] \ar@{->}@<-2pt>[dr] \ar@{.>}[rr]    &      &{P_{3}} \ar@{->}@<2pt>[dr] \ar@{->}@<-2pt>[dr]  \ar@{.>}[rr]    &     &{P_{5}} \ar@{->}@<2pt>[dr] \ar@{->}@<-2pt>[dr]             \\
&  &{P_{2}} \ar@{->}@<2pt>[ur] \ar@{->}@<-2pt>[ur] \ar@{.>}[rr]         &  &{P_{4}} \ar@{->}@<2pt>[ur] \ar@{->}@<-2pt>[ur] \ar@{.>}[rr]       &  &{P_{6}}               &{\cdots}
}
$$
where dotted arrows indicate Auslander-Reiten translation $\tau_1^{-}$.

Observed that $P_{1}$ is the unique simple projective $\Lambda$-module and $P_{3}=\tau^{-}_{1}(P_{1})$.
By \cite[Section 2.2]{MY16}, $T_{\Lambda}=P_{3} \oplus P_{2}$ is an $1$-APR tilting $\Lambda$-module associated with $P_{1}$.
The $1$-APR tilt $\End_{\Lambda}(T_{\Lambda})^{op}$ is isomorphic to algebra $\Lambda$.
The projective module $P_{1}\otimes P_{1} =\Gamma e_{1,1}$ corresponding to the vertex $(1,1)$ is the unique simple projective $\Gamma$-module.
Let $Q_{\Gamma}=(\Lambda \otimes \Lambda) / (P_{1} \otimes P_{1})$. Therefore, by Theorem \ref{the-tensor-products-of-n-APR-tilting-modules}, $T_{\Gamma}=(P_{3}\otimes P_{3})\oplus Q_{\Gamma}$ is an $2$-APR tilting $\Gamma$-module associated with $P_{1}\otimes P_{1}$,
here $P_{3}\otimes P_{3}=\tau^{-}_{1}(P_{1})\otimes \tau^{-}_{1}(P_{1})=\tau^{-}_{2}(P_{1} \otimes P_{1})$.
The $2$-APR tilt $\End_{\Gamma}(T_{\Gamma})^{op}$ is also an $2$-representation infinite algebra,
and its bound quiver is given as follows
$$
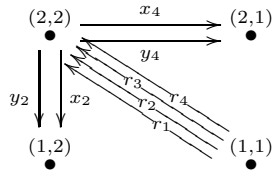
\xymatrix@C=0.8cm@R1.0cm{
   &{\stackrel{(2,2)}{\bullet}} \ar@{->}@<4pt>[rr]^{x_{4}} \ar@{->}@<-2pt>[rr]_{y_{4}}  \ar@{->}@<4pt>[d]^{x_{2}} \ar@{->}@<-4pt>[d]_{y_{2}}
   &      &{\stackrel{(2,1)}{\bullet}}   \\
&{\stackrel{(1,2)}{\bullet}}   &      &{\stackrel{(1,1)}{\bullet}}\ar@{->}@<6pt>[ull]|(.4){r_{1}} \ar@{->}@<2pt>[ull]|(.5){r_{2}} \ar@{->}@<-2pt>[ull]|(.6){r_{3}}  \ar@{->}@<-6pt>[ull]|(.4){r_{4}}
}
$$
with relations
$y_{2}r_{1}+y_{2}r_{2}=0,~y_{2}r_{3}+y_{2}r_{4}=0,~y_{4}r_{1}+y_{4}r_{3}=0$ and $y_{4}r_{2}+y_{4}r_{4}=0$.
\end{example}

\end{document}